\documentclass[a4paper,12pt]{amsart}
\usepackage[utf8]{inputenc}
\usepackage{hyperref,tikz}
\usetikzlibrary{arrows}%,decorations.pathmorphing,shapes.geometric,
\usepackage[headings]{fullpage}
\usepackage{amsmath, amssymb}
\allowdisplaybreaks[1]
\usepackage[alphabetic, nobysame, initials]{amsrefs}
\usepackage{graphicx}

\newtheorem{thm}{Theorem} %[section]
\newtheorem{lem}[thm]{Lemma}

\newtheorem{cor}[thm]{Corollary}
\theoremstyle{definition}
\newtheorem{defn}{Definition}

\newcommand{\Ne}{\mathbb N}
\newcommand{\NN}{\mathcal N}
\renewcommand{\Re}{\mathbb R}
\newcommand{\Red}{\Re^d}
\newcommand{\Sed}{\mathbb S^{d-1}}
\renewcommand{\epsilon}{\varepsilon}
\renewcommand{\phi}{\varphi}

\newcommand{\pack}[2]{P\left(#1,#2\right)}

\newcommand{\norm}[1]{\left\lVert#1\right\rVert}
\newcommand{\abs}[1]{\left\lvert#1\right\rvert}
\newcommand{\setbuilder}[2]{\{#1:#2\}}

\newcommand{\bd}{\operatorname{bd}}
\newcommand{\inter}{\operatorname{int}}
\newcommand{\vol}{\operatorname{vol}}

\newcommand{\un}[1]{#1^{\wedge}}
\newcommand{\card}[1]{\left\lvert#1\right\rvert}
\newcommand{\epsi}{\varepsilon}

\title{Arrangements of homothets of a convex body}
\author[M. Nasz\'odi, J. Pach, K. J. Swanepoel]{M\'arton Nasz\'odi \and J\'anos 
Pach \and 
Konrad Swanepoel}
\address{Department of Geometry, Lorand E\"otv\"os University, Pazm\'any 
P\'eter S\'etany 1/C Budapest, Hungary 1117}\email{marton.naszodi@math.elte.hu}
\address{EPFL Lausanne and R\'enyi Institute, Budapest}\email{pach@cims.nyu.edu}
\address{Department of Mathematics, London School of Economics and Political 
Science, Houghton Street, London WC2A 2AE, United 
Kingdom}\email{k.swanepoel@lse.ac.uk}
\subjclass[2010]{}
\thanks{M\'arton Nasz\'odi acknowledges the support of the J\'anos Bolyai 
Research Scholarship of the Hungarian Academy of Sciences, and
the National Research, Development, and Innovation Office, NKFIH 
Grants PD104744 and K119670. Part of this paper was 
written when Swanepoel visited EPFL in April 2015. Research by J\'anos Pach was 
supported in part by Swiss National Science Foundation grants 200020-144531 and 
200020-162884.}
\begin{document}
\begin{abstract}
Answering a question of F\"uredi and Loeb (1994), we show that the maximum 
number of pairwise intersecting homothets of a $d$-dimensional centrally 
symmetric convex body~$K$, none of which contains the center of another in its interior,
is at most $O(3^d d\log d)$.
If $K$ is not necessarily centrally symmetric and the role of its center is played by its centroid, then the above bound can be replaced by $O(3^d\binom{2d}{d}d\log d)$.
We establish analogous results for the case where the center is defined as an arbitrary point in the interior of $K$.
We also show that in the latter case, one can always find families of at least $\Omega((2/\sqrt{3})^d)$ translates of $K$ with the above property.
\end{abstract}
\maketitle

\section{Introduction}\label{sec:intro}

A \emph{convex body} $K$ in the $d$-dimensional Euclidean space $\Red$ is a 
compact convex set with non-empty interior, and is \emph{$o$-symmetric} if $K=-K$.
A (positive) \emph{homothet} of $K$ is a set 
of the form $\lambda K+v:=\setbuilder{\lambda k+v}{k\in K}$, where $\lambda>0$ is the 
homothety ratio, and $v\in\Red$ is a translation vector.
We investigate 
arrangements of homothets of convex bodies. The starting point of our 
investigations is Problem~4.4 of a paper of F\"uredi and Loeb~\cite{FL94}:
\begin{quote}\itshape
Is it true that for any 
centrally symmetric body $K$ of dimension $d, d\geq d_0$, the number of 
pairwise intersecting homothetic copies of $K$ which do not contain each 
other's centers is at most $2^d$? 
\end{quote}
There exist $8$ such homothets of 
the circular disc \cites{MM92, HJLM93} (Fig.~\ref{circles}).
\begin{figure}[t]
\centering
\definecolor{qqzzqq}{rgb}{0.,0.6,0.}
\definecolor{zzttqq}{rgb}{0.6,0.2,0.}
\definecolor{ccqqqq}{rgb}{0.8,0.,0.}
\definecolor{qqqqff}{rgb}{0.,0.,1.}
\begin{tikzpicture}[line cap=round,line join=round,>=triangle 45,scale=10, 
thick]
\clip(6,0.27) rectangle (6.8,0.95);
  \draw [color=ccqqqq,fill opacity=0.25] 
(6.5401326852420025,4.1895795177247654) circle (3.5276227841994934cm);
  \draw [color=ccqqqq,fill opacity=0.25] (9.904397308362515,1.5591310036741608) 
circle (3.516484136725196cm);
  \draw [color=ccqqqq,fill opacity=0.25] (8.414186133020543,-2.437344421106587) 
circle (3.5179875872955026cm);
  \draw [color=ccqqqq,fill opacity=0.25] (4.1241,-2.392) circle (3.735cm);
  \draw [color=ccqqqq,fill opacity=0.25] 
(2.9726574473021294,1.7397626612913697) circle (3.5690422480878397cm);
  \draw [color=qqzzqq,fill=qqzzqq,fill opacity=0.1] 
(6.503814035722377,0.6615416297673377) circle (0.15341799622043542cm);
  \draw [color=qqzzqq,fill=qqzzqq,fill opacity=0.1] 
(6.359978517495642,0.6030252918603594) circle (0.1548796357153899cm);
  \draw [color=qqzzqq,fill=qqzzqq,fill opacity=0.1] 
(6.482572911772887,0.5077182480629958) circle (0.15487792157398633cm);
  \draw [fill=qqqqff] (6.503814035722377,0.6615416297673377) circle (0.1pt);
  \draw [fill=qqqqff] (6.359978517495642,0.6030252918603594) circle (0.1pt);
  \draw [fill=qqqqff] (6.482572911772887,0.5077182480629958) circle (0.1pt);
\end{tikzpicture}
\caption{A pairwise intersecting strict Minkowski arrangement of $8$ circles 
(after Harary et al.\ \cite{HJLM93})}\label{circles}
\end{figure}
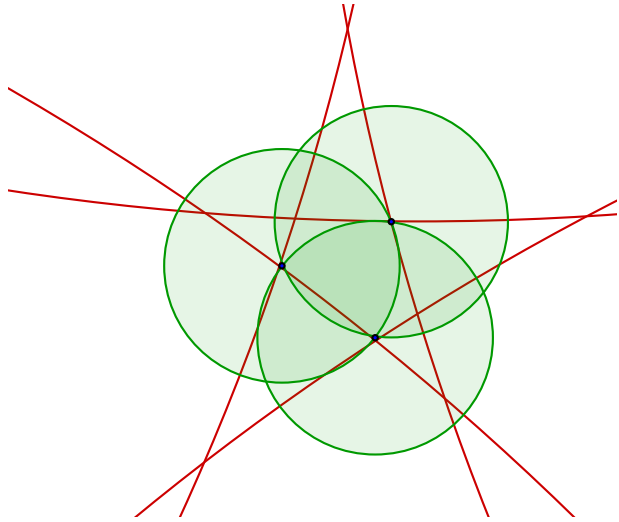
A \emph{Minkowski arrangement} of an $o$-symmetric convex body $K$ is defined to be a family $\{v_i+\lambda_i K\}$ 
of positive homothets of $K$ such that none of the homothets contains the center of any other homothet in its interior. This notion was introduced by L. Fejes T\'oth \cite{FT65} in the context of Minkowski's fundamental theorem on the minimal determinant of a packing lattice for a symmetric convex body, and further studied in the papers \cites{FT67, FT99, BSz04}, and in connection to the Besicovitch covering theorem in \cite{FL94}.
Recently, Minkowski arrangements have been used to study a problem arising in the design of wireless networks~\cite{NSS17}.

We also define a \emph{strict Minkowski arrangement} of $K$ 
to be a family $\{v_i+\lambda_i K\}$ of positive homothets of $K$
such that none of the homothets contains the center of any other homothet.
We write $\kappa(K)$ ($\kappa'(K)$) for the largest number of homothets that a pairwise intersecting (strict) 
Minkowski arrangement of $K$ can have.

Thus, the question of F\"uredi and Loeb may be phrased as follows: \emph{Is it true 
that 
$\kappa'(K)\leq 2^d$ for any $o$-symmetric convex body $K$ in $\Red$ with 
$d$ sufficiently large?}
A construction of Talata \cite{Ta05} implies that the answer to this question is negative: there exists a $d$-dimensional convex body $K$ such that $\kappa'(K)\geq \frac{16}{35}\sqrt{7}^d$ for all $d\geq 3$
(see Section~\ref{section:Hadwiger}).
The question now becomes to find an upper bound for $\kappa'(K)$.
It follows from \cite{FL94}*{Theorem 2.1} that $\kappa'(K)\leq\kappa(K)\leq 5^d$ for any $o$-symmetric $d$-dimensional convex body $K$.
Our first result is the following improvement.
\begin{thm}\label{thm1}
For any $d$-dimensional $o$-symmetric convex body $K$,
 \[
  \kappa^\prime(K)\leq \kappa(K)\leq O(3^d d\log d).
 \]
\end{thm}
It is easy to see that for the $d$-cube $C^d$, $\kappa(C^d)= 3^d$, which shows that 
the upper bound for $\kappa(K)$ in
Theorem~\ref{thm1} is sharp up to the $O(d\log d)$ factor.
We will in fact prove a strengthening of this theorem in Theorem~\ref{thm:kappauppersymm} below.
However, we have no better upper bound for $\kappa'(K)$ than that for~$\kappa(K)$.

Theorem~\ref{thm1} implies that if we have a sequence of balls $B_1,B_2,\ldots,B_n$ (of not necessarily equal radii) in a $d$-dimensional normed space, such that for any $1\leq i<j\leq n$, the center of $B_j$ is on the boundary of $B_i$, then $n\leq O(6^d d^2\log d)$ (Corollary~\ref{cor:chainofballs} in Section~\ref{sec:nonsymmkappa}). 
This has recently been improved to an almost tight bound by Polyanskii \cite{Polyanskii}.
This result has an application to $k$-distance sets~\cite{Swanepoel}.

We next consider convex bodies that are not 
necessarily $o$-symmetric, and extend the notion of Minkowski arrangement as follows.
In the absence of a center, we choose a fixed reference point interior to the convex body.
\begin{defn}
Let $K$ be a convex body and $p$ a fixed point in the interior of $K$. 
A \emph{Minkowski arrangement of $K$ with respect to $p$} is a 
family $\{v_i+\lambda_i K\}$ 
of positive homothets of $K$ with the property that $v_i+p$ is not in 
$v_j+\lambda_j\inter(K)$, for any distinct $i$ and $j$.
We denote the largest number of homothets that a pairwise intersecting 
Minkowski arrangement of $K$ with respect to $p$ can have by~$\kappa(K,p)$.

Similarly, we define a \emph{strict Minkowski arrangement of $K$ with respect 
to $p$} 
to be a family $\{v_i+\lambda_i K\}$ of positive homothets of $K$
such that $v_i+p\notin v_j+\lambda_j K$, for any $i\neq j$, and we write 
$\kappa'(K,p)$ for the largest number of homothets that a pairwise intersecting strict 
Minkowski arrangement of $K$ with respect to $p$ can have.
\end{defn}
Thus, when $K$ is $o$-symmetric, $\kappa(K)=\kappa(K,o)$ and $\kappa'(K)=\kappa'(K,o)$.
For bodies that are not $o$-symmetric, we also need to measure in some way how far they are from being $o$-symmetric.
\begin{defn}
Let $K$ be a convex body with $p$ in its interior.
Define $\theta(K,p)$, the \emph{measure of asymmetry of $K$ with respect to 
$p$} to be $\theta(K,p):=\inf\setbuilder{\theta}{p-K\subseteq\theta(K-p)}$.
\end{defn}

(Gr\"unbaum \cite{Gru63}*{Section~6.1} defines a quantity similar to $\theta$.)
Our next result generalizes~Theorem~\ref{thm1}.
\begin{thm}\label{thm2}
Let $K$ be a convex body in $\Red$ with $p\in\inter(K)$. Then 
\[\kappa'(K,p)\leq\kappa(K,p)\leq\left(\frac{3}{2}\right)^d\frac{\vol(K-K)}{
\vol((K-p)\cap (p-K))} O(d(\log d + \log\theta(K,p))).\]

If $c$ is the centroid of $K$ then \[\kappa'(K,c)\leq\kappa(K,c)\leq 3^d\binom{2d}{d} O(d\log d).\]
\end{thm}
There exists a $d$-dimensional convex body $K$ 
with centroid $c$ such that $\kappa(K,c)\geq\sqrt{10}^d$ (detailed at the end of Section~\ref{section:packing}).
We prove a strengthening of Theorem~\ref{thm2} in Theorem~\ref{thm:kappauppernonsymm} below.

When $K$ is $o$-symmetric, Arias-de Reyna, Ball, and Villa \cite{ABV98}*{Theorem~1} derived a lower bound $\Omega((2/\sqrt{3})^d)$ for 
the strict Hadwiger number $H'(K)$ (see Definition~\ref{defn:Hadwiger} in Section~\ref{section:Hadwiger}), which implies that $\kappa'(K) = \Omega((2/\sqrt{3})^d)$.
We show the same lower bound in the non-symmetric case.
\begin{thm}\label{thm:kappalowerbdassym}
 Let $K$ be a convex body with $p\in\inter(K)$. Then 
 $\kappa'(K,p)>c(2/\sqrt{3})^d$
 for some universal constant $c>0$.
\end{thm}

We also prove a lower bound for a variant $h'(K)$ (see Definition~\ref{defn:Hadwiger2} in Section~\ref{section:packing}) 
of the strict Hadwiger number $H'(K)$ for $K$ that is not $o$-symmetric.

\begin{thm}\label{thm:hlower}
 Let $K$ be a convex body in $\Red$ with $o\in\inter(K)$. Then for sufficiently large $d$,
 \[
\kappa'(K,o)
\geq
h'(K)\geq 
\frac{1}{4d^2}\left(\frac{2}{\sqrt{3}}\right)^d.
 \]
\end{thm}

The paper is organized as follows.  In Section~\ref{section:Hadwiger}, we 
apply a result of Talata to give a negative answer to the question of F\"uredi and 
Loeb quoted at the beginning of the Introduction. In Section~\ref{section:packing}, we state two stronger versions of Theorem~\ref{thm1}
(Theorems~\ref{thm:kappauppersymm} and \ref{thm:kappauppernonsymm}). The latter one, which is the main result in this paper, is valid for all (not necessarily centrally symmetric) convex bodies. It is proved in Section~\ref{sec:nonsymmkappa}.
The other two main results, Theorems~\ref{thm:kappalowerbdassym} and 
\ref{thm:hlower} also hold for non-symmetric bodies. They are proved in the last section. Along the way, we 
obtain some facts (Lemmas~\ref{lem:nsbowarrow} and \ref{lem:cpl}, 
Theorem~\ref{thm:ABVnonsymmetric}) that are useful in studying non-symmetric 
convex bodies in general.

\section{A negative answer to the question of F\"uredi and Loeb}\label{section:Hadwiger}
Let $K$ be an $o$-symmetric convex body in $\Red$.
Denote the norm with unit ball $K$ by~$\norm{\cdot}_K$.
\begin{defn}\label{defn:Hadwiger}
For any convex body $K$, the \emph{Hadwiger number} 
(resp., \emph{strict Hadwiger number}) of $K$ is defined as the maximum number 
$H(K)$ (resp., $H'(K)$) of 
non-overlapping (resp., disjoint) translates of $K$ touching $K$.
\end{defn}
When $K$ is $o$-symmetric, $H(K)$ equals the maximum number of points 
$v_1,\dots,v_m$ such that $\norm{v_i}_K=1$ for all $i$ and 
$\norm{v_i-v_j}_K\geq 1$ for all distinct $i,j$.
Then
$\{K\}\cup\setbuilder{K+v_i}{i=1,\dots,m}$ is a Minkowski arrangement of translates of 
$K$ all intersecting in $o$, hence $\kappa(K)\geq H(K)+1$.
Similarly, $H'(K)$ equals the maximum number of points 
$v_1,\dots,v_m$ such that $\norm{v_i}_K=1$ for all $i$ and 
$\norm{v_i-v_j}_K > 1$ for all distinct $i,j$.
Thus, $\setbuilder{K+v_i}{i=1,\dots,m}$ is a strict Minkowski arrangement of translates of 
$K$ all intersecting in $o$, hence $\kappa'(K)\geq H'(K)$.
To answer the question of F\"uredi and Loeb in the negative, it is therefore sufficient to find an $o$-symmetric convex body $K$ with 
$H'(K)>2^d$.
In dimension $3$, we may take the Euclidean ball $B^3$, for which it is well 
known that $H'(B^3)=12$.
For $d>3$, we may use a result of Talata \cite{Ta05}*{Lemma~3.1} according to which 
$H'(C^k\times K)=2^kH'(K)$ holds for any $o$-symmetric convex body $K$, where 
$C^k$ is the $k$-dimensional cube.
In particular, $H'(B^3\times C^{d-3})=3\cdot 2^{d-1} > 2^d$ for all $d\geq 3$.
In fact, Talata \cite{Ta05}*{Theorem~1.3} constructed $d$-dimensional $o$-symmetric convex 
bodies $K$ such that $H'(K) \geq \frac{16}{35}\sqrt{7}^d$ for all $d\geq 3$.
It follows that for these bodies, $\kappa(K)\geq\kappa'(K)\geq\Omega(\sqrt{7}^d)$.

\section{Packing and non-symmetric norms}\label{section:packing}
\begin{defn}
If $K$ is $o$-symmetric, we define the \emph{packing number} 
$\pack{K}{\lambda}$ of 
$K$ as the maximum number of points in the normed space with unit ball $K$, 
such that 
the ratio of the maximal distance to the minimal distance is at most $\lambda$.
We denote the normed space with unit ball $K$ as 
$\NN$, and also use the notations $\kappa(\NN)$, $\pack{\NN}{\lambda}$, $H(\NN)$, \dots 
in place of $\kappa(K)$, $\pack{K}{\lambda}$, $H(K)$, \dots
\end{defn}

It follows from the isodiametric inequality in normed spaces (an immediate 
corollary to the Brunn--Minkowski Theorem \cites{B47, M63}) that
\begin{equation}\label{eq:packingbound}
 \pack{\NN}{\lambda}\leq \left(\lambda+1\right)^d
\end{equation}
for any $d$-dimensional normed space $\NN$. 
(See Lemma~\ref{lemma:P} below for a generalization.)
Our next result strengthens Theorem~\ref{thm1}.
\begin{thm}\label{thm:kappauppersymm}
Let $\NN$ be a $d$-dimensional real normed space. Then
 \[
  \kappa^\prime(\NN)\leq \kappa(\NN)\leq 
\pack{\NN}{2(1+\tfrac{1}{d})}(d+O(1))\log d= O(3^d d\log d).
 \]
\end{thm}
Since $\kappa(C^d)\geq H(C^d)+1 = 3^d$, which shows that 
the upper bound for $\kappa(\NN)$ in
Theorem~\ref{thm:kappauppersymm} is sharp up to the $O(d\log d)$ factor.
Theorem~\ref{thm:kappauppersymm} is a special case of 
Theorem~\ref{thm:kappauppernonsymm} below that also deals with $K$ that are not 
necessarily $o$-symmetric, considered next.
\begin{defn}
If the convex body $K$ contains the origin in the interior, we define the (asymmetric) norm 
$\norm{\cdot}_K\colon\Red\to\Re$ by 
$\norm{x}_K=\inf\setbuilder{\lambda>0}{x\in\lambda K}$.
\end{defn}
Note that the measure of asymmetry of $K$ with respect to $o$ can be defined in terms of the norm:
\[\theta(K,o)=\sup\setbuilder{\norm{x}_K/\norm{-x}_K}{x\in\bd K}.\]
We need the following well-known result.
\begin{lem}[Minkowski \cite{Mi1897}]\label{lem:asym}
For any $d$-dimensional convex body $K$ with centroid $c$, $\theta(K,c)\leq d$.
\end{lem}
We will also use the (symmetric) norm defined by the unit ball $K\cap-K$.
Thus, $\norm{x}_{K\cap-K} = \max\{\norm{x}_K,\norm{-x}_K\}$.
We also need another symmetric norm.
\begin{defn}
For any convex body $K$, define its \emph{central symmetral} to be 
$\frac12(K-K)$.
If $o\in\inter(K)$, then $\pack{K}{\lambda}$ is defined to be the maximum 
number of points $p_1,\dots,p_m$ such that 
\[\frac{\max\setbuilder{\norm{p_i-p_j}_{\frac12(K-K)}}{1\leq i<j\leq m}}{\min\setbuilder{\norm{p_i-p_j}_{K\cap -K}}{1\leq i<j\leq m}}\leq\lambda.\]
\end{defn}
If $K$ is $o$-symmetric, then the norms $\norm{\cdot}_{K}$, 
$\norm{\cdot}_{K\cap -K}$, and $\norm{\cdot}_{\frac12(K-K)}$ are all identical, and 
$\pack{K}{\lambda}$ coincides with the definition given before.
\begin{lem}\label{lemma:P}
For any convex body $K$ with $o$ in its interior and any $\lambda>0$,
\[ \pack{K}{\lambda} \leq 
(\lambda+1)^d\frac{\vol(\frac12(K-K))}{\vol(K\cap-K)}.\]
\end{lem}

We also need to generalize the Hadwiger number to the non-symmetric case, in 
the following non-standard way.
\begin{defn}\label{defn:Hadwiger2}
If $o\in\inter(K)$, define $h(K)$ to be the maximum number of points 
$p_1,\dots,p_m$ on $\bd K$ such that $\norm{p_i-p_j}_K\geq 1$ for all distinct 
$i,j=1,\dots,m$.
Similarly, we define $h'(K)$ to be the maximum number of points 
$p_1,\dots,p_m\in\bd K$ such that $\norm{p_i-p_j}_K > 1$ for all distinct 
$i,j=1,\dots,m$.
\end{defn}

If $K=-K$, then $h(K)=H(K)$ and $h'(K)=H'(K)$ (cf.\ 
Definition~\ref{defn:Hadwiger}).
This is not necessarily the case if $K$ is not $o$-symmetric.
(Note that for all convex bodies, $H(K)=H(\frac12(K-K))$.)
Generalizing our observation for the symmetric case above, if 
$p_1,\dots,p_m\in\bd K$ satisfy $\norm{p_i-p_j}_K > 1$ for all distinct $i,j$, 
then the collection $\setbuilder{K-p_i}{i=1,\dots,m}$ is a pairwise 
intersecting strict Minkowski arrangement of translates of $K$, hence 
$\kappa'(K,o)\geq h'(K)$.
Similarly (by adding $K$ to the collection) we have $\kappa(K,o)\geq h(K)+1$.
We can now formulate our generalization of Theorem~\ref{thm2}.
\begin{thm}\label{thm:kappauppernonsymm}
Let $K$ be a convex body in $\Red$ with $o\in\inter(K)$. Then 
\[\kappa'(K,o)\leq\kappa(K,o)\leq 
\pack{K}{2(1+\tfrac{1}{d})}(d+O(1))(\log d + \log\theta(K,o)).\]
If $c$ is the centroid of $K$ then \[\kappa(K,c)\leq 
\pack{K}{2(1+\tfrac{1}{d})}(2d+O(1))\log d.\]
\end{thm}
The proof is postponed to the next section.
The following is an example of a $d$-dimensional convex body $K$ with centroid $c$ for which 
$\kappa(K,c)$ is much larger than in the symmetric case.
Note that $\kappa(\Delta,o)=10$, where $\Delta$ is a triangle with centroid $o$ 
\cite{FT95} (see Fig.~\ref{triangles}).
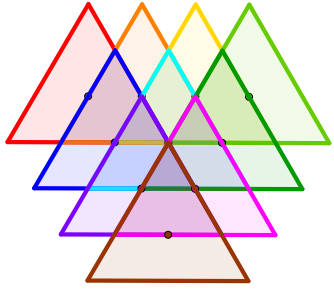
\begin{figure}
\centering
\definecolor{ffqqqq}{rgb}{1.,0.,0.}
\definecolor{zzttqq}{rgb}{0.6,0.2,0.}
\definecolor{ffqqff}{rgb}{1.,0.,1.}
\definecolor{xfqqff}{rgb}{0.4980392156862745,0.,1.}
\definecolor{qqzzqq}{rgb}{0.,0.6,0.}
\definecolor{qqffff}{rgb}{0.,1.,1.}
\definecolor{qqqqff}{rgb}{0.,0.,1.}
\definecolor{wwccqq}{rgb}{0.4,0.8,0.}
\definecolor{ffdxqq}{rgb}{1.,0.8431372549019608,0.}
\definecolor{ffxfqq}{rgb}{1.,0.4980392156862745,0.}
\definecolor{uuuuuu}{rgb}{0.26666666666666666,0.26666666666666666,
0.26666666666666666}
\begin{tikzpicture}[line cap=round,line join=round,>=triangle 45,scale=0.87]
  \fill[line width=1.6pt,color=ffqqqq,fill=ffqqqq,fill opacity=0.1] 
(2.963416172889948,4.077454144465436) -- (1.740007440834372,1.9750915660786712) 
-- (4.1724112077108515,1.9667698141002143) -- cycle;
  \draw [line width=1.6pt,color=ffqqqq] (2.963416172889948,4.077454144465436)-- 
(1.740007440834372,1.9750915660786712);
  \draw [line width=1.6pt,color=ffqqqq] 
(1.740007440834372,1.9750915660786712)-- 
(4.1724112077108515,1.9667698141002143);
  \draw [line width=1.6pt,color=ffqqqq] 
(4.1724112077108515,1.9667698141002143)-- (2.963416172889948,4.077454144465436);
  \draw [fill=uuuuuu] (2.9586116071450586,2.67310517488144) circle (1.5pt);
  \fill[line width=1.6pt,color=ffxfqq,fill=ffxfqq,fill opacity=0.1] 
(3.7742174285154424,4.074680227139284) -- 
(2.5508086964598653,1.972317648752519) -- 
(4.983212463336347,1.9639958967740623) -- cycle;
  \draw [line width=1.6pt,color=ffxfqq] 
(3.7742174285154424,4.074680227139284)-- (2.5508086964598653,1.972317648752519);
  \draw [line width=1.6pt,color=ffxfqq] 
(2.5508086964598653,1.972317648752519)-- (4.983212463336347,1.9639958967740623);
  \draw [line width=1.6pt,color=ffxfqq] 
(4.983212463336347,1.9639958967740623)-- (3.7742174285154424,4.074680227139284);
  \draw [fill=ffxfqq] (3.769412862770551,2.6703312575552878) circle (1.5pt);
  \fill[line width=1.6pt,color=ffdxqq,fill=ffdxqq,fill opacity=0.1] 
(4.585018684140937,4.071906309813132) -- (3.361609952085359,1.9695437314263666) 
-- (5.794013718961841,1.96122197944791) -- cycle;
  \draw [line width=1.6pt,color=ffdxqq] (4.585018684140937,4.071906309813132)-- 
(3.361609952085359,1.9695437314263666);
  \draw [line width=1.6pt,color=ffdxqq] 
(3.361609952085359,1.9695437314263666)-- (5.794013718961841,1.96122197944791);
  \draw [line width=1.6pt,color=ffdxqq] (5.794013718961841,1.96122197944791)-- 
(4.585018684140937,4.071906309813132);
  \draw [fill=ffdxqq] (4.580214118396046,2.667557340229136) circle (1.5pt);
  \fill[line width=1.6pt,color=wwccqq,fill=wwccqq,fill opacity=0.1] 
(5.395819939766431,4.0691323924869796) -- 
(4.1724112077108515,1.9667698141002143) -- 
(6.6048149745873355,1.958448062121758) -- cycle;
  \draw [line width=1.6pt,color=wwccqq] 
(5.395819939766431,4.0691323924869796)-- 
(4.1724112077108515,1.9667698141002143);
  \draw [line width=1.6pt,color=wwccqq] 
(4.1724112077108515,1.9667698141002143)-- 
(6.6048149745873355,1.958448062121758);
  \draw [line width=1.6pt,color=wwccqq] 
(6.6048149745873355,1.958448062121758)-- (5.395819939766431,4.0691323924869796);
  \draw [fill=wwccqq] (5.39101537402154,2.664783422902983) circle (1.5pt);
  \fill[line width=1.6pt,color=qqqqff,fill=qqqqff,fill opacity=0.1] 
(3.3664145178302496,3.373892701010362) -- 
(2.143005785774673,1.2715301226235973) -- 
(4.575409552651153,1.2632083706451405) -- cycle;
  \draw [line width=1.6pt,color=qqqqff] 
(3.3664145178302496,3.373892701010362)-- (2.143005785774673,1.2715301226235973);
  \draw [line width=1.6pt,color=qqqqff] 
(2.143005785774673,1.2715301226235973)-- (4.575409552651153,1.2632083706451405);
  \draw [line width=1.6pt,color=qqqqff] 
(4.575409552651153,1.2632083706451405)-- (3.3664145178302496,3.373892701010362);
  \draw [fill=qqqqff] (3.361609952085359,1.9695437314263666) circle (1.5pt);
  \fill[line width=1.6pt,color=qqffff,fill=qqffff,fill opacity=0.1] 
(4.1772157734557425,3.37111878368421) -- (2.953807041400167,1.2687562052974455) 
-- (5.386210808276647,1.2604344533189884) -- cycle;
  \draw [line width=1.6pt,color=qqffff] (4.1772157734557425,3.37111878368421)-- 
(2.953807041400167,1.2687562052974455);
  \draw [line width=1.6pt,color=qqffff] 
(2.953807041400167,1.2687562052974455)-- (5.386210808276647,1.2604344533189884);
  \draw [line width=1.6pt,color=qqffff] 
(5.386210808276647,1.2604344533189884)-- (4.1772157734557425,3.37111878368421);
  \draw [fill=qqffff] (4.1724112077108515,1.9667698141002143) circle (1.5pt);
  \fill[line width=1.6pt,color=qqzzqq,fill=qqzzqq,fill opacity=0.1] 
(4.988017029081239,3.368344866358058) -- 
(3.7646082970256596,1.2659822879712928) -- 
(6.197012063902143,1.2576605359928363) -- cycle;
  \draw [line width=1.6pt,color=qqzzqq] (4.988017029081239,3.368344866358058)-- 
(3.7646082970256596,1.2659822879712928);
  \draw [line width=1.6pt,color=qqzzqq] 
(3.7646082970256596,1.2659822879712928)-- 
(6.197012063902143,1.2576605359928363);
  \draw [line width=1.6pt,color=qqzzqq] 
(6.197012063902143,1.2576605359928363)-- (4.988017029081239,3.368344866358058);
  \draw [fill=qqzzqq] (4.983212463336347,1.9639958967740623) circle (1.5pt);
  \fill[line width=1.6pt,color=xfqqff,fill=xfqqff,fill opacity=0.1] 
(3.769412862770551,2.6703312575552878) -- 
(2.5460041307149743,0.5679686791685236) -- 
(4.978407897591454,0.5596469271900668) -- cycle;
  \draw [line width=1.6pt,color=xfqqff] 
(3.769412862770551,2.6703312575552878)-- 
(2.5460041307149743,0.5679686791685236);
  \draw [line width=1.6pt,color=xfqqff] 
(2.5460041307149743,0.5679686791685236)-- 
(4.978407897591454,0.5596469271900668);
  \draw [line width=1.6pt,color=xfqqff] 
(4.978407897591454,0.5596469271900668)-- (3.769412862770551,2.6703312575552878);
  \draw [fill=xfqqff] (3.7646082970256596,1.2659822879712928) circle (1.5pt);
  \fill[line width=1.6pt,color=ffqqff,fill=ffqqff,fill opacity=0.1] 
(4.580214118396046,2.667557340229136) -- (3.356805386340468,0.5651947618423713) 
-- (5.789209153216949,0.5568730098639146) -- cycle;
  \draw [line width=1.6pt,color=ffqqff] (4.580214118396046,2.667557340229136)-- 
(3.356805386340468,0.5651947618423713);
  \draw [line width=1.6pt,color=ffqqff] 
(3.356805386340468,0.5651947618423713)-- (5.789209153216949,0.5568730098639146);
  \draw [line width=1.6pt,color=ffqqff] 
(5.789209153216949,0.5568730098639146)-- (4.580214118396046,2.667557340229136);
  \draw [fill=ffqqff] (4.575409552651153,1.2632083706451405) circle (1.5pt);
  \fill[line width=1.6pt,color=zzttqq,fill=zzttqq,fill opacity=0.1] 
(4.1724112077108515,1.9667698141002143) -- 
(2.9490024756552757,-0.13559276428655045) -- 
(5.381406242531756,-0.14391451626500693) -- cycle;
  \draw [line width=1.6pt,color=zzttqq] 
(4.1724112077108515,1.9667698141002143)-- 
(2.9490024756552757,-0.13559276428655045);
  \draw [line width=1.6pt,color=zzttqq] 
(2.9490024756552757,-0.13559276428655045)-- 
(5.381406242531756,-0.14391451626500693);
  \draw [line width=1.6pt,color=zzttqq] 
(5.381406242531756,-0.14391451626500693)-- 
(4.1724112077108515,1.9667698141002143);
  \draw [fill=zzttqq] (4.167606641965961,0.5624208445162191) circle (1.5pt);
\end{tikzpicture}
\caption{A pairwise intersecting Minkowski arrangement of $10$ triangles 
\cite{FT95}}\label{triangles}
\end{figure}
A Cartesian product of $\lfloor d/2\rfloor$ triangles gives a $d$-dimensional convex body $K$ 
with centroid $c$ such that $\kappa(K,c)\geq 10^{\lfloor d/2\rfloor}\geq\sqrt{10}^{d-1}$.

\section{\texorpdfstring{Bounding $\kappa$ from above}{Bounding K from 
above}}\label{sec:nonsymmkappa}

\begin{proof}[Proof of Lemma~\ref{lemma:P}]
Let $T\subset\Red$ be such that $\norm{x-y}_{K\cap-K}\geq 1$ and $\norm{x-y}_{\frac12(K-K)}\leq\lambda$
for all distinct $x,y\in T$.
Then $\setbuilder{v+\frac12(K\cap-K)}{v\in T}$ is a packing.
Let $P=T+\frac12(K\cap-K)$.
Then $\vol(P)=2^{-d}\card{T}\vol(K\cap-K)$ and 
\[P-P=T-T+(K\cap-K)\subseteq\frac{\lambda}{2}(K-K)+\frac12(K-K)=\frac{\lambda+1}{2
}(K-K).\]
By the Brunn--Minkowski inequality, $\vol(P-P)\geq 2^d\vol(P)$, and it follows 
that 
\[\card{T}=\frac{2^d\vol(P)}{\vol(K\cap-K)}\leq\frac{\vol(P-P)}{\vol(K\cap-K)}
\leq\frac{(\lambda+1)^d\vol(\frac12(K-K))}{\vol(K\cap-K)}.\qedhere\]
\end{proof}

Before we prove Theorem~\ref{thm:kappauppernonsymm}, we first show an extension 
of the so-called ``bow-and-arrow'' inequality of \cite{FL94} 
(Corollary~\ref{cor:bowarrow} below) to the case of an asymmetric norm.
\begin{defn}
For any non-zero $v\in\Red$ write $\un{v}=\frac{1}{\norm{v}_K}v$ for the 
normalization of $v$ with respect to $\norm{\cdot}_K$.
\end{defn}
We will only consider normalizations with respect to $\norm{\cdot}_K$, 
never with respect to $\norm{\cdot}_{K\cap-K}$ or $\norm{\cdot}_{\frac12(K-K)}$.

\begin{lem}\label{lem:nsbowarrow}
Let $K$ be a convex body in $\Red$ containing $o$ in its interior.
Let $a,b\in\Red$ such that $\norm{a}_K\geq\norm{b}_K>0$.
Then
\[ 
\norm{\un{a}-\un{b}}_K\geq\frac{\norm{a-b}_K-\norm{a}_K+\norm{b}_K}{\norm{b}_K}
.\]
\end{lem}

\begin{proof}
\begin{align*}
\norm{a-b}_K &=\norm{\norm{a}_K\un{a}-\norm{b}_K\un{b}}_K\\
&=\norm{\norm{b}_K(\un{a}-\un{b})+(\norm{a}_K-\norm{b}_K)\un{a}}_K\\
&\leq \norm{b}_K\norm{\un{a}-\un{b}}_K+\norm{a}_K-\norm{b}_K. \qedhere
\end{align*}
\end{proof}

\begin{cor}\label{cor:bowarrow}
For any two non-zero elements $a$ and $b$ of a normed space,
\[ 
\norm{\un{a}-\un{b}}\geq\frac{\norm{a-b}-\abs{\norm{a}-\norm{b}}}{\norm{b}}.\]
\end{cor}

\begin{proof}[Proof of Theorem~\ref{thm:kappauppernonsymm}]
Consider a pairwise intersecting Minkowski arrangement $\setbuilder{\lambda_i K+v_i}{i=1,\dots,m}$.
Without loss of generality, $\lambda_1=\min_i\lambda_i=1$ and $v_1=o$.
Given $N\in\Ne$ and $\delta>0$, we partition the Minkowski arrangement into $N$ 
subarrangements as follows.
Let $I_j=\setbuilder{i}{\lambda_i\in[(1+\delta)^{j-1},(1+\delta)^{j})}$ for each 
$j=1,\dots,N$, and let $I_\infty=\setbuilder{i}{\lambda_i\in[(1+\delta)^N,\infty)}$.
We bound the size of each subarrangement $\setbuilder{\lambda_i K+v_i}{i\in I_j}$, 
$j\in\{1,\dots,N,\infty\}$, separately.
Finally, we choose appropriate values for $N$ and $\delta$.

The next lemma bounds $I_j$, $j\neq\infty$, in terms of $\delta$ and $K$.
\begin{lem}\label{lem:I}
Let $K$ be a $d$-dimensional convex body with $o\in\inter(K)$.
Let $\setbuilder{v_i+\lambda_i K}{i\in I}$ be a pairwise intersecting Minkowski 
arrangement of positive homothets of $K$, with $\lambda_i\in[1,1+\delta)$ for 
each $i\in I$.
Then \[\card{I}\leq \pack{K}{2(1+\delta)}.\]
\end{lem}

\begin{proof}
%Write $T=\setbuilder{v_i}{i\in I}$.
For any distinct $i,j\in I$, $(v_i+\lambda_i K)\cap (v_j+\lambda_j K)\neq \emptyset$, so there exist $x,y\in K$ such that $v_i-v_j=\lambda_j y-\lambda_i x$.
Since $\lambda_i,\lambda_j\in[1, 1+\delta]$, $o\in K$ and $K$ is convex, $\lambda_i x,\lambda_j y\in (1+\delta)K$.
Hence, $v_i-v_j\in(1+\delta)(K-K)$ and $\norm{v_i-v_j}_{\frac12(K-K)}\leq 
2(1+\delta)$.
Since $v_i\notin v_j+\lambda_j\inter(K)$, it follows that 
$v_i-v_j\notin\inter(K\cap-K)$ for all distinct $i,j\in I$, which gives 
$\norm{v_i-v_j}_{K\cap-K}\geq 1$.
\end{proof}

The following lemma is used to bound $I_\infty$.
\begin{lem}\label{lem:Iinfty}
Let $K$ be a $d$-dimensional convex body with $o\in\inter(K)$.
Let $\setbuilder{v_i+\lambda_i K}{i\in I}$ be a Minkowski arrangement of 
positive homothets of $K$ with $\lambda_i\geq 1$, $(v_i+\lambda_i K)\cap-\epsi 
K\neq\emptyset$ and $o\notin v_i+\lambda_i\inter(K)$ for all $i\in I$.
Then
\[ \card{I}\leq \pack{K}{\frac{2}{1-\epsi}}.\]
\end{lem}

We first consider any two homothets in the Minkowski arrangement of the previous 
lemma.
\begin{lem}\label{lem:Iinftylocal}
Let $v_1+\lambda_1 K$ and $v_2+\lambda_2 K$ be two positive homothets of $K$ 
such that $\lambda_1,\lambda_2\geq 1$, $v_1\notin v_2+\lambda_2\inter(K)$, 
$v_2\notin v_1+\lambda_1\inter(K)$, $o\notin v_i+\lambda_i\inter(K)$ and 
$(v_i+\lambda_i K)\cap-\epsi K\neq\emptyset$ \textup{(}$i=1,2$\textup{)}.
Then 
$\norm{\frac{1}{\norm{-v_1}_K}(-v_1)-\frac{1}{\norm{-v_2}_K}(-v_2)}_{K\cap-K}
\geq1-\epsi$.
\end{lem}

\begin{proof}
Since $\norm{\cdot}_{K\cap-K}$ is symmetric, we may assume that 
$\norm{-v_1}_K\leq\norm{-v_2}_K$.
Since $(v_1+\lambda_1 K)\cap-\epsi K\neq\emptyset$, $v_1+\lambda_1 x=-\epsi y$ 
for some $x,y\in K$.
Therefore, 
$\norm{-v_1}_K\leq\lambda_1\norm{x}_K+\epsi\norm{y}_K\leq\lambda_1+\epsi$.
Also, since $o\notin v_1+\lambda_1\inter(K)$, we have that 
$\norm{-v_1}_K\geq\lambda_1$.
Similarly, $\lambda_2\leq \norm{-v_2}_K\leq\lambda_2+\epsi$, and it follows from $v_1\notin v_2+\lambda_2\inter(K)$ that 
$\norm{v_1-v_2}_K\geq\lambda_2$.
We apply Lemma~\ref{lem:nsbowarrow} to obtain
\begin{align*}
\norm{\un{(-v_1)}-\un{(-v_2)}}_{K\cap-K}&\geq \norm{\un{(-v_2)}-\un{(-v_1)}}_K\\
&\geq \frac{\norm{v_1-v_2}_K-\norm{-v_2}_K+\norm{-v_1}_K}{\norm{-v_1}_K}\\
&\geq \frac{\lambda_2-(\lambda_2+\epsi)+\norm{-v_1}_K}{\norm{-v_1}_K}\\
&= 1 - \frac{\epsi}{\norm{-v_1}_K} \geq  1 - \frac{\epsi}{\lambda_1} \geq 
1-\epsi. \qedhere
\end{align*}
\end{proof}

\begin{proof}[Proof of Lemma~\ref{lem:Iinfty}]
For each $i\in I$, let $t_i=\un{(-v_i)}$.
Let $T:=\setbuilder{t_i}{i\in I}$.
By Lemma~\ref{lem:Iinftylocal}, $\norm{t_i-t_j}_{K\cap-K}\geq1-\epsi$ for all 
distinct $i,j\in I$.
Since $T\subset \bd K\subset K$, $\norm{t_i-t_j}_{\frac12(K-K)}\leq 2$.
It follows that $\card{I}\leq \pack{K}{2/(1-\epsi)}$.
\end{proof}

\medskip\noindent
We now finish the proof of Theorem~\ref{thm:kappauppernonsymm}.
By Lemma~\ref{lem:I}, $\card{I_j}\leq \pack{K}{2(1+\delta)}$ for $j=1,\dots,N$,
and by Lemma~\ref{lem:Iinfty} applied to $I_\infty$ and 
$\epsi=\theta(K,o)(1+\delta)^{-N}$, 
\[\card{I_\infty}\leq \pack{K}{\frac{2}{1-\theta(K,o)(1+\delta)^{-N}}}.\]
It follows that
\[ m = \sum_{j=1}^N\card{I_j} + \card{I_\infty} \leq 
N\pack{K}{2(1+\delta)}+\pack{K}{\frac{2}{1-\theta(K,o)(1+\delta)^{-N}}}.\]
We now choose 
\[N:=1+\left\lceil\frac{\log d + \log\theta(K,o)}{\log(1+\frac{1}{d})}\right\rceil = 
(d+O(1))O(\log d + \log\theta(K,o))\]
and $\delta=1/d$.
Then \[N\geq 1+\frac{\log d + \log\theta(K,o)}{\log(1+\delta)},\] which implies 
that \[\frac{2}{1-\theta(K,o)(1+\delta)^{-N}}\leq 2(1+\delta),\] hence
\renewcommand{\qedsymbol}{$\blacksquare$}
\[ m\leq 
\pack{K}{2(1+\tfrac{1}{d})}(N+1)=\pack{K}{2(1+\tfrac{1}{d})}(d+O(1))(\log d + 
\log\theta(K,o)).\]
The second inequality follows from the first and Lemma~\ref{lem:asym}.
\end{proof}
Note that Theorem~\ref{thm:kappauppersymm} immediately follows from Theorem~\ref{thm:kappauppernonsymm}, and Theorem~\ref{thm1} from Theorem~\ref{thm:kappauppersymm}. 
\begin{proof}[Proof of Theorem~\ref{thm2}]
The first statement follows from Theorem~\ref{thm:kappauppernonsymm} and  Lemma~\ref{lemma:P}.
Also, by a result of Milman and Pajor \cite{MiPa00}*{Corollary~3} for a convex 
body $K$ with centroid $o$, $\vol(K)/\vol(K\cap-K)\leq 2^d$, which, together 
with the Rogers--Shephard inequality \cite{RoS57}
$\vol(K-K)/\vol(K)\leq\binom{2d}{d}$, gives the second statement.
\end{proof}

We derive the following application of Theorem~\ref{thm1}.
\begin{cor}\label{cor:chainofballs}
Let $K$ be an $o$-symmetric convex body, and $p_1,p_2,\ldots,p_n$ be points in 
$\Red$. Let $r_1,r_2,\ldots,r_n >0$, and assume that for any $1\leq i<j\leq n$, 
we have that $p_j\in p_i+r_i\bd K$. Then $n\leq O(6^d d^2\log d)$.
\end{cor}

\begin{proof}%[Proof of Corollary~\ref{cor:chainofballs}]
Let $D\subseteq\{1,2,\dots,n-1\}$ be the index set of a longest decreasing subsequence of $r_1,r_2,\dots,r_{n-1}$. Thus, if $i,j\in D$ with $i<j$, then $r_i\geq r_j$.
Then $\setbuilder{p_i+r_i K}{i\in D}$ is a pairwise intersecting Minkowski arrangement, and by Theorem~\ref{thm1}, $\card{D}=O(3^d d\log d)$.

Next, let $I\subseteq\{1,2,\dots,n-1\}$ be the index set of a longest increasing subsequence of $r_1,r_2,\dots,r_{n-1}$, that is, if $i,j\in I$ with $i< j$, then $r_i\leq r_j$.
Let $m:=\min I$.
By the triangle inequality we have $r_i\leq 2r_m$ for any $i\in I$.
Indeed, without loss we may assume $i\neq m$, and then, since $m<i<n$,
$\norm{p_i-p_m}_K=\norm{p_n-p_m}_K=r_m$ and $\norm{p_n-p_i}_K=r_i$, from which $r_i\leq 2r_m$ follows.

We now use the 
same ``logarithmic cut'' method as in the proof of 
Theorem~\ref{thm:kappauppernonsymm}.
Choose $N\in\Ne$.
For each $k=1,\dots,N$, let $I_k:=\setbuilder{i\in I}{r_i/r_m\in[2^{(k-1)/N},2^{k/N})}$.
Then $\norm{p_i-p_j}_K\in[2^{(k-1)/N},2^{k/N})$ for any distinct $i,j\in I_k$, hence \[\card{I_k}\leq P(K,2^{1/N})\leq(1+2^{1/N})^d\] by \eqref{eq:packingbound}, and \[\card{I}=\sum_{k=1}^N\card{I_k}\leq N(1+2^{1/N})^d.\]
We now choose an optimal value $N:=d$ to obtain $\card{I}\leq d(1+2^{1/d})^d=O(2^d d)$.

By the Erd\H{o}s--Szekeres Theorem \cite{ESz35}, any sequence of real numbers for which all decreasing subsequences are of length at most $s$ and all increasing subsequences are of length at most $t$, has length at most $st$.
It follows that $n-1\leq\card{D}\cdot\card{I}$, hence
\[n = O(3^d d \log d)O(2^d d)=O(6^d d^2\log d),\]
as claimed.
\end{proof}

\section{\texorpdfstring{Bounding $\kappa'$ from below}{Bounding K' from 
below}}\label{sec:lowerbds}
In this section we prove Theorems~\ref{thm:kappalowerbdassym} and \ref{thm:hlower}, by extending a lower bound for the strict Hadwiger number $H'(K)$ by Arias-de Reyna, Ball, and Villa \cite{ABV98}*{Theorem~1} to non-symmetric convex bodies.
Earlier, Bourgain  \cite{FL94} showed an exponential 
lower bound to $H'(K)$ for $o$-symmetric $K$ that depends only on the dimension of $K$.
(This argument was also independently discovered by Talata \cite{T98}.)
The key tool used by Bourgain and Talata is Milman's Quotient of Subspace Theorem 
(or, in short, the QS Theorem) \cite{Mi85}. 

In order to obtain a lower bound on $\kappa(K,p)$ in the non-symmetric case, 
it is possible to use a non-symmetric 
version of the QS Theorem (see Milman and Pajor \cite{MiPa00} or Rudelson \cite{Rudelson}),
or one may generalize the approach from \cite{ABV98}.
The first approach does not lead to a concrete lower bound, and we will follow the second.
However, neither approach is straightforward. One obstacle is that 
$p$ may not coincide with the centroid of $K$.
To bypass this problem, we use the following topological result.
\begin{lem}%[``Centroid of Projection Lemma'']
\label{lem:cpl}
 Let $K$ be a convex body in $\Red$. Then there is a $(d-1)$-dimensional 
linear subspace $H$ of $\Red$ such that the centroid of the orthogonal 
projection of $K$ onto $H$ is the origin.
\end{lem}
Statements similar to this lemma are known (see for instance \cite{Iz14}).
The lemma itself is surely also known.
However, since we could not find a reference, we include its simple proof.
\begin{proof}%[Proof of Lemma~\ref{lem:cpl}.]
For any unit vector $u\in\Sed$, let $f(u)$ be the centroid of the orthogonal 
projection of $K$ onto $u^\perp$.
We need to show that $f(u)=o$ for some $u\in\Sed$.
Suppose not.
Then $\un{f}\colon\Sed\to\Sed$ defined by $\un{f}(u)=f(u)/\norm{f(u)}_2$ is a 
continuous, even mapping such that $\langle u, \un{f}(u)\rangle=0$ for all 
$u\in\Sed$.
Since $\un{f}$ is even, its degree is even (see for instance 
\cite{Ha02}*{Proposition~2.30}).
Also, $\un{f}(u)\neq -u$ for all $u\in\Sed$.
It follows that $\un{f}$ is homotopic to the 
identity map, which has degree $1$, a contradiction.
\end{proof}
We briefly outline how this lemma can be combined with the non-symmetric 
QS Theorem to obtain that $h^{\prime}(K)>c^d$ for some universal constant $c>1$.
Later on in this section, we will obtain more explicit bounds 
(Theorems~\ref{thm:kappalowerbdassym} and \ref{thm:hlower}) using the main 
result of \cite{ABV98}.

First, the non-symmetric version of the QS Theorem (\cite{MiPa00}*{Theorem~9} and \cite{Rudelson}*{Theorem~4}), combined with Lemma~\ref{lem:cpl} yields that 
for any convex body $K$ in $\Red$, there is a roughly $(d/2)$-dimensional 
subspace $E$ and an origin centered ellipsoid ${\mathcal E}$ in $E$, such that 
for an appropriate projection $P$ of $\Red$, we have 
${\mathcal E} \subseteq P(K)\cap E \subseteq c{\mathcal E}$ with some universal 
constant $c$.
By a theorem of Milman \cite{M71} (see also \cite{MS86}*{Section~4.3}), we can 
take a $C(d/2)$-dimensional subspace $E^\prime$ of $E$ such that
${\mathcal E}^\prime:={\mathcal E}\cap E^\prime \subseteq P(K)\cap 
E^\prime \subseteq 1.1{\mathcal E}^\prime$, for 
a universal constant $C>0$. 
Although this is stated only for symmetric bodies $K$ in \cite{MS86}, the proof 
works in the non-symmetric case as well.

Now, we can follow the proof of the symmetric case (Theorem~4.3) in \cite{FL94} 
closely. 
There are exponentially many points on the relative boundary of 
$K^\prime:=P(K)\cap E^\prime$ such that the distance (with respect to the 
slightly non-symmetric norm $\norm{\cdot}_{K^{\prime}}$ on $E^\prime$) 
between any two points is at least $1.21$. 
Let $X$ be the set of these points. For every $x\in X$, choose a point $p\in 
\bd K$ such that $P(p)=x$. 
These points satisfy Definition~\ref{defn:Hadwiger2}.

Before we prove Theorems~\ref{thm:kappalowerbdassym} and \ref{thm:hlower}, we state a non-symmetric version of \cite{ABV98}*{Theorem~1}.
\begin{thm}\label{thm:ABVnonsymmetric}
Let $K$ be a convex body in $\Red$ with $o\in\inter(K)$.
Let $\mu$ denote the uniform \textup{(}with respect to Lebesgue measure\textup{)} probability 
measure on $K$. 
Then, for any $0<t<\sqrt{2}$,
\[
F(t) := \mu\otimes \mu\setbuilder{(x,y)\in K\times K}{\norm{x-y}_K \leq t}
\leq\left(\frac{t^2(4-t^2)}{4}\right)^{d/2}.
\]
\end{thm}
The proof of Theorem~\ref{thm:ABVnonsymmetric} is virtually the same as in 
\cite{ABV98}.
We recall the first part of this proof, which is where the only (slight) difference lies.
In that proof, $\mu\otimes \mu\setbuilder{(x,y)\in K\times 
K}{\norm{x-y}_K \leq t}$ is written as a threefold convolution.
In the non-symmetric case it is easy to see that
for any $t\geq 0$,
\[
F(t) = \mu\otimes \mu\setbuilder{(x,y)\in K\times 
K}{\norm{x-y}_K \leq t}=(\chi_K\ast \chi_{-K} \ast \chi_{tK})(0),
\]
where $\chi_A$ denotes the indicator function of a set $A$. 
The only difference with the symmetric case is the occurrence of $\chi_{-K}$ instead of $\chi_K$ in the right-hand side.
This does not affect the rest of the proof in \cite{ABV98}, which is 
an application of a 
strong form of Young's inequality to this threefold convolution, and which 
we do not repeat.

\begin{proof}[Proof of Theorem~\ref{thm:kappalowerbdassym}]
Let $m=\frac{1}{4}(2/\sqrt{3})^d$, and choose $2m$ points $x_1,\dots,x_{2m}$ 
independently and uniformly from $K$.
Then the expected number of ordered pairs $(x_i,x_j)$, $i\neq j$, such that 
$\norm{x-y}_K\leq 1$, equals $2m(2m-1)F(1)$ by linearity of expectation.
This quantity is at most $m$ by Theorem~\ref{thm:ABVnonsymmetric} and the 
choice of $m$.
Thus, there exists a choice of points $x_1,\dots,x_{2m}$ from $K$ such that 
$\norm{x_i-x_j}_K > 1$ for all except at most $m$ pairs $(i,j)$, $i\neq j$.
For each such pair, delete one of the points.
We end up with $m$ points $x_1,\dots,x_m$, say, such that $\norm{x_i-x_j}_K> 1$ 
for all distinct $i,j$.
It follows that 
$-x_1+K,\ldots,-x_m+K$ is a strict Minkowski arrangement.
This family of translates of $K$ is clearly pairwise intersecting, since all 
members contain the origin.
\end{proof}
If $K$ is $o$-symmetric, it follows from Corollary~\ref{cor:bowarrow} that if $a,b\in K$ satisfy $\norm{a-b}_K>1$, then $\norm{\un{a}-\un{b}}_K>1$.
We therefore obtain the lower bound $h'(K)=H'(K)=\Omega((2/\sqrt{3})^d)$ by normalizing the points $x_i$ in the proof above.
When $K$ is not $o$-symmetric, the probabilistic argument above can be adapted to obtain the slightly worse lower bound $h'(K)=\Omega((2/\sqrt{3})^d/d^2)$ of Theorem~\ref{thm:hlower}.
The proof is technically more involved, and the details are as follows.
\begin{proof}[Proof of Theorem~\ref{thm:hlower}]
We first assume that $o$ is the centroid of $K$.
By Lemma~\ref{lem:asym}, $\theta(K,o)\leq d$, hence $\norm{-x}_K\leq d\norm{x}_K$ for all $x\in\Red$.
Let $k\in\Ne$, to be fixed later,
and 
choose $k$ points $x_1,\dots,x_{k}$ independently and uniformly from $K$.
Let $\delta$ be such that $e^{\delta d} = (d+4)/(d+1)$ (thus, $\delta\sim 
3/d^2$).
Then the expected number of points $x_i$ such that $\norm{x_i}_{K}\leq 
1-\delta$ 
(we call these points \emph{short}) equals 
\[(1-\delta)^d k < e^{-\delta d}k = \biggl(1-\frac{3}{d+4}\biggr)k.\]
We say that an ordered pair $(x_i,x_j)$, $i\neq j$, is \emph{close} if 
$\norm{x_i-x_j}_K \leq 1 + (d+1)\delta$.
Then by Theorem~\ref{thm:ABVnonsymmetric}, the expected number of ordered close 
pairs is less than
\begin{align*}
&\mathrel{\phantom{=}} 
k^2\biggl(\frac{(1+(d+1)\delta)^2(4-(1+(d+1)\delta)^2)}{4}\biggr)^{d/2}\\
&< k^2\biggl(\frac{\sqrt{3}}{2}\biggr)^d
\biggl(1+\frac43(d+1)\delta\biggr)^{d/2}\\
&< k^2\biggl(\frac{\sqrt{3}}{2}\biggr)^d e^{2\delta d(d+1)/3} = 
k^2\biggl(\frac{\sqrt{3}}{2}\biggr)^d \biggl(\frac{d+4}{d+1}\biggr)^{2(d+1)/3} \\
&< e^2 k^2 \biggl(\frac{\sqrt{3}}{2}\biggr)^d.
\end{align*}
Thus, if we delete each short point, as well as one member of each close pair, 
then the expected number $m$ of points left is at least
\[
m\geq
k\biggl(\frac{3}{d+4} - \biggl(\frac{\sqrt{3}}{2}\biggr)^d e^2 k\biggr).
\]
To maximize this quadratic expression in $k$, we set $k = 
\frac12(2/\sqrt{3})^d\frac{3}{e^2(d+4)}$.
Thus, there exist at least
\[ m \geq \biggl(\frac{2}{\sqrt{3}}\biggr)^d\frac{9}{4e^2(d+4)^2} \]
points
$x_1,\ldots,x_m\in K$ such that $\norm{x_i}> 1-\delta$ for each $i$ and 
$\norm{x_i-x_j}_K > 1+(d+1)\delta$ for each pair of distinct $i,j$. 
Normalize these points to obtain
\[\un{x}_1=\frac{1}{\norm{x_1}_{K}}x_1, \dots, 
\un{x}_m=\frac{1}{\norm{x_m}_{K}}x_m\in\bd K.\]
Note that $\norm{\un{x}_i-x_i}_K<\delta$, hence $\norm{x_i-\un{x}_i}_K < 
d\delta$.
By the triangle inequality, for distinct $i,j$ we have
\begin{align*}
    \norm{\un{x}_i-\un{x}_j}_K &\geq \norm{x_i-x_j}_K - \norm{x_i-\un{x}_i}_K - 
\norm{\un{x}_j-x_j}_K \\
    &> 1+(d+1)\delta - d\delta - \delta = 1.
\end{align*}
Therefore, the requirements of Definition~\ref{defn:Hadwiger2} are satisfied.

Next, we reduce the case when $o$ is an arbitrary point in $\inter(K)$ to the 
case where $o$ is the centroid of $K$.
By Lemma~\ref{lem:cpl}, there is an orthogonal projection $\pi\colon\Red\to H$ 
where $H$ is a linear subspace 
of dimension $d-1$ such that $\pi(K)$ has centroid $o$.
By what was shown above, there exist $m\geq 
\left(2/\sqrt{3}\right)^{d-1}9/(4e^2(d+3)^2)$ points 
$y_1,\dots,y_m\in\bd \pi(K)$ such that $\norm{y_1-y_j}_{\pi(K)}> 1$ for all 
distinct $i,j$.
For each $i$, choose $p_i\in\pi^{-1}(y_i)\in K$.
Since $\norm{x}_K\geq\norm{\pi(x)}_{\pi(K)}$ for all $x\in\Red$, we obtain that 
$p_1,\ldots,p_m\in\bd K$ satisfy $\norm{p_i-p_j}_K\geq\norm{y_1-y_j}_{\pi(K)}> 
1$ for all distinct $i,j$.
This finishes the proof of Theorem~\ref{thm:hlower} in the general case.
\end{proof}

%\bibliographystyle{amsalpha}
%\bibliography{biblio}
% \bib, bibdiv, biblist are defined by the amsrefs package.

\end{document}